\documentclass [11pt]{amsart}
\usepackage{amsmath}
\usepackage[foot]{amsaddr}
\usepackage{pdfsync}
\usepackage{amsthm}
\usepackage{graphicx}
\usepackage{epsfig}
\usepackage{caption}
\usepackage{tikz}
\usepackage{amsfonts, amssymb}
\usepackage[utf8]{inputenc} %% mac coding
\usepackage[english]{babel}
\usepackage[numbers]{natbib}
\usepackage{mathbbol}
\usepackage{tcolorbox}

\usepackage{hyperref}
\usepackage{bm}
\usepackage[mathscr]{eucal}
\usepackage{tikz-cd}
\usepackage{mathtools}
\usepackage{stmaryrd}
\usepackage{xfrac}
\usepackage{graphicx}

\usepackage{cancel}
\usepackage{soul}

\usepackage[left=2.0cm,right=2.0cm,top=2.0cm,bottom=2.0cm]{geometry}

\newtheorem{theorem}{Theorem}
\newtheorem{lemma}[theorem]{Lemma}
\newtheorem{proposition}[theorem]{Proposition}
\newtheorem{corollary}[theorem]{Corollary}
\theoremstyle{definition}
\newtheorem{remark}[theorem]{Remark}
\theoremstyle{definition}
\newtheorem{example}[theorem]{Example}
\theoremstyle{definition}
\newtheorem{definition}[theorem]{Definition}

\setcitestyle{aysep={ }}

\def\AA{{\mathcal A}}

\def\MM{{\mathcal M}}
\def\MMM{\overline{\mathcal{M}}}

\DeclareMathOperator{\D}{\bf{D}}
\def\AAA{{ \mathscr{A}}}
\DeclareMathOperator{\SO}{SO}
\DeclareMathOperator{\PSL}{PSL}
\DeclareMathOperator{\flip}{flip}

\DeclareMathOperator{\Ann}{Ann}
\def\PP{{\mathscr P}_{q}}

\def\bl{\textbf{bl}}

\def\MMMM{{\mathcal M}_{0, \lowercase{n}}}

\begin{document}
\title[Compactifications of $\MMMM$ associated with Alexander self dual complexes]{Compactifications of $\MMMM$ associated with Alexander self dual complexes: Chow ring, $\psi$-classes, and intersection numbers}
\author{Ilia Nekrasov$^1$}
\email{geometr.nekrasov@yandex.ru}
\address{$^1$Chebyshev laboratory, Department of Math. and Mech., St. Petersburg State University}
\author{Gaiane Panina$^2$}
\email{gaiane-panina@rambler.ru}
\address{$^2$PDMI RAS, St. Petersburg State University}

\begin{abstract}
An  \textit{Alexander self-dual complex} gives rise to a compactification of $\mathcal{M}_{0,n}$, called \textit{ ASD compactification}, which is a smooth algebraic variety. ASD compactifications include (but are not exhausted by) the \textit{polygon spaces}, or the moduli spaces of flexible polygons. We present an explicit description of the Chow rings of ASD compactifications. We study the analogs of Kontsevich's tautological bundles, compute their Chern classes,  compute top intersections of the Chern classes, and derive a recursion for the intersection numbers.
\end{abstract}

\maketitle


%\tableofcontents

\section{Introduction}

The moduli space  of $n$-punctured rational curves $\MM_{0,n}$ and its various compactifications  is a classical object, bringing together algebraic geometry, combinatorics, and topological robotics. Recently, D.I.Smyth \cite{Smyth13} classified all  \textit{modular compactifications} of $\MM_{0,n}$. We make use of an interplay between different compactifications, and:

\begin{itemize}
\item describe the classification in terms of (what we call) \textit{preASD simplicial complexes};

\item describe the Chow rings of the  compactifications arising from Alexander self-dual complexes (ASD compactifications);

\item compute for ASD compactifications the associated Kontsevich's \textit{$\psi$-classes}, their top monomials,  and give a recurrence relation for the top monomials.
\end{itemize}

Oversimplifying, the main approach is as follows. Some (but not all) compactifications are the  well-studied \textit{polygon spaces}, that is, moduli spaces of flexible polygons. A polygon space corresponds to a \textit{threshold} Alexander self-dual complex. Its cohomology ring (which equals the Chow ring) is known due to  J.-C. Hausmann and A. Knutson \cite{HKn}, and  A. Klyachko\cite{Kl}. The paper \cite{NP} gives a computation-friendly presentation of the ring. Due to Smyth \cite{Smyth13}, all the modular compactifications correspond to \textit{preASD complexes}, that is, to those complexes that are contained in an ASD complex. A removal of a facet of a preASD complex amounts to a blow up  of the associated compactification. Each  ASD compactification is achievable from a threshold ASD compactification by a sequence of blow ups  and blow downs. Since the changes in the Chow ring are controllable, one can start with a polygon space, and then (by elementary steps) reach any of the ASD compactifications and describe its Chow ring (Theorem \ref{ASDChow}).

M. Kontsevich's $\psi$-classes \cite{Kon} arise here in a standard way. Their computation of  is a mere modification of the Chern number count for the tangent bundle over $\mathbb{S}^2$ (a classical exercise in  a topology course). The recursion  (Theorem \ref{ThmRecursion}) and the top monomial counts (Theorem \ref{main_theorem}) follow.

 It is worthy mentioning that a disguised  compactification by simple games, i.e., ASD complexes, is discussed from a combinatorial viewpoint in \cite{PanGal}.

\bigskip

Now let us give a very brief overview of moduli compactifications of $\MM_{0,n}$. A compactification by a smooth variety is very desirable since it makes intersection theory applicable. We also expect that
(1) a compactification is \textit{modular}, that is, itself  is the moduli space  of some curves and marked points lying on it, and (2) 
 the complement of $\MM_{0,n}$ (the “boundary”) is a  divisor.

The space  $\MM_{0,n}$ is viewed as the configuration space of $n$ distinct marked points (``particles'') living in the complex projective plane. The space $\MM_{0,n}$ is non-compact due to forbidden collisions of the marked points. Therefore, each compactification should suggest an answer to the question: what happens when two (or more) marked points  tend  to each other?  There exist two possible choices: either one allows some (not too many!) points to coincide, either one applies a blow up. It is important that the blow ups amount to adding points that correspond to $n$-punctured nodal curves of arithmetic genus zero.

A compactification obtained by blow ups only is the celebrated \textit{Deligne--Mumford compactification}. If one avoids blow ups and allows (some carefully chosen) collections of points to coincide, one gets an ASD-compactification; among them are the \textit{polygon spaces}. Diverse combinations of these two options (in certain cases one allows points to collide, in other cases one applies a blow up) are also possible; the complete classification is due to \cite{Smyth13}.

Now let us be more precise and look at the compactifications in more detail.

\subsection{Deligne--Mumford compactification}\label{DMcom}

\begin{definition}\cite{DM69}
Let $B$ be a scheme. A \textit{family of rational nodal curves with $n$ marked points} over $B$ is
\begin{itemize}
\item a flat proper morphism $\pi: C \rightarrow B$ whose geometric fibers $E_{\bullet}$ are nodal connected curves of arithmetic genus zero, and
\item a set of sections $(s_{1}, \dots, s_{n})$ that do not intersect nodal points of geometric fibers. \newline
In this language, the sections correspond to marked points. The above condition means that a nodal point of a curve may not be marked.
\end{itemize}
A family $(\pi:C \rightarrow B; s_{1}, \dots, s_{n})$ is \textit{stable} if the divisor $K_{\pi}+s_{1}+\dots+s_{n}$ is $\pi$-relatively ample.
\medskip

Let us rephrase this condition: a family $(\pi:C \rightarrow B; s_{1}, \dots, s_{n})$ is  {stable} if each irreducible component of each geometric fiber has at least three \textit{special points} (nodal points and points of the sections $s_{i}$). 
\end{definition}

\begin{theorem} \cite{DM69}
(1) There exists a smooth and proper over $\mathbb{Z}$ stack $\MMM_{0,n}$, representing the moduli functor of stable rational curves. Corresponding moduli scheme is a smooth projective variety over $\mathbb{Z}$.

(2) The compactification equals the  moduli space for $n$-punctured stable curves of arithmetic genus zero with $n$ marked points. A stable  curve is a curve of arithmetic genus zero   with at worst nodal singularities and finite automorphism group. This means that (i) every irreducible component has at least three marked or nodal points, and (ii) no marked point is a nodal point. 
\end{theorem}

The Deligne-Mumford compactification has  a natural stratification by stable trees with $n$ leaves. A \textit{stable tree with $n$ leaves} is a tree with exactly $n$ leaves enumerated by elements of $[n]=\{1,...,n\}$ such that each vertice is at least trivalent. 

Here and in the sequel, we use the following \textbf{notation}: by \textit{vertices} of a tree we mean all the vertices (in the usual graph-theoretical sense) excluding the leaves. A \textit{bold edge} is an edge connecting two vertices  (see  Figure \ref{Figtrees}).
 
 \begin{figure}
 \includegraphics[scale=0.7]{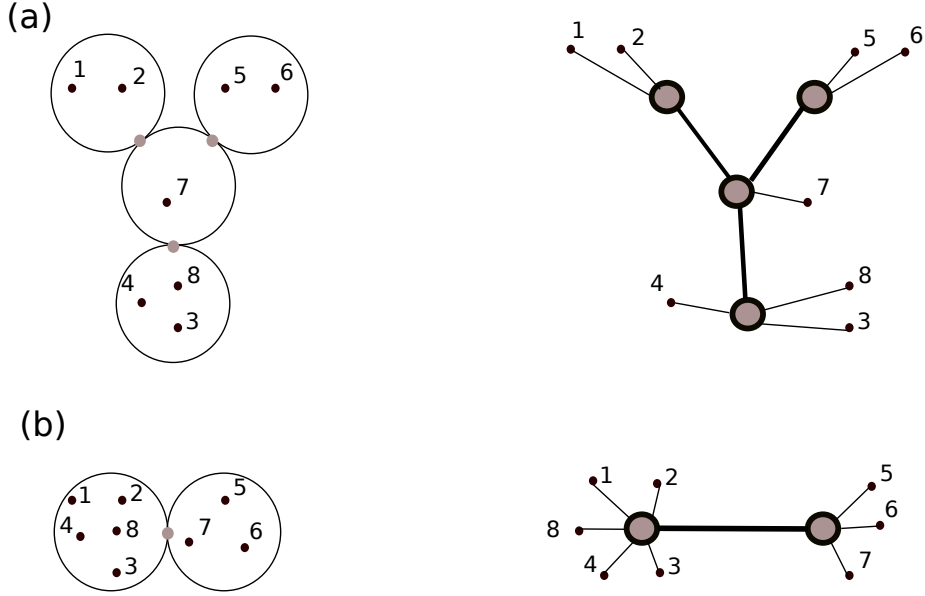}
\caption{Stable nodal curves (left) and the corresponding trees (right)}
\label{Figtrees}
\centering
\end{figure}

The initial space $\MM_{0,n}$ is a stratum corresponding to the one-vertex tree. Two-vertex trees (Fig.\ref{Figtrees}(b)) are in a bijection with bipartitions of the set $[n]$: $T\sqcup T^{c} = [n]$ s.t. $|T|, |T^{c}| > 1$. We denote the closure of the corresponding stratum  by $\D_{T}$. The latter are important since the (Poincaré duals of) closures of the strata generate the Chow ring $\textbf{A}^{*}(\MMM_{0,n})$:

\begin{theorem}\cite[Theorem 1]{Keel92} \label{KeelPresentation}
The Chow ring $\bf{A}^{*}$$(\MMM_{0,n})$ is isomorphic to the polynomial ring
$$\mathbb{Z}\left[\D_{T}:\; T \subset [n]; |T|>1, |T^{c}| >1\right]$$
factorized by the  relations:
\begin{enumerate}
\item $\D_{T} = \D_{T^{c}}$;
\item $\D_{T}\D_{S} =0$ unless $S \subset T$ or $T \subset S$ or $S \subset T^{c}$ or $T \subset S^{c}$;
\item For any distinct elements $i,j,k,l \in [n]$:
$$\sum_{i,j \in T; k,l \not\in T} \D_{T} = \sum_{i,k \in T; j,l \not\in T} \D_{T} = \sum_{i,l \in T; j,k \not\in T} \D_{T}$$
\end{enumerate}
\end{theorem}

\subsection{Weighted compactifications}\label{Weicom}

The next breakthrough step was done by B. Hassett in \cite{Has03}. 

Define a \textit{weight data} as an element $\mathcal{A} = (a_{1}, \dots, a_{n}) \in \mathbb{R}^{n}$ such that
\begin{itemize}
\item $0 < a_{i} \leq 1$ for any $i \in [n]$,

\item $a_{1}+ \dots +a_{n} > 2$.
\end{itemize}

\begin{definition}
Let $B$ be a scheme. A family of nodal curves with $n$ marked points $(\pi: C \rightarrow B; s_{1}, \dots, s_{n})$ is \textit{$\mathcal{A}$--stable} if
\begin{enumerate}
\item $K_{\pi}+ a_{1}s_{1}+\dots + a_{n}s_{n}$ is $\pi$-relatively ample,

\item whenever the sections $\{s_{i}\}_{i \in I}$ intersect for some $I \subset [n]$, one has $\sum_{i \in I} a_{i} < 1$.
\end{enumerate}
The first condition can be rephrased as:
 each irreducible component of any geometric fiber has at least three {distinct} special points.

\end{definition}

\begin{theorem}\cite[Theorem 2.1]{Has03}
For any weight data $\mathcal{A}$ there exist a connected Deligne--Mumford stack $\MMM_{0, \mathcal{A}}$ smooth and proper over $\mathbb{Z}$, representing the moduli functor of $\mathcal{A}$--stable rational curves. The corresponding moduli scheme is a smooth projective variety over $\mathbb{Z}$.
\end{theorem}

The  Deligne--Mumford compactification arises as a special case  for the weight data $ (1,\dots, 1)$.

It is natural to ask: how much does a weighted compactification $\MMM_{0, \mathcal{A}}$ depend on $\mathcal{A}$? Pursuing this question, let us consider the  space of parameters:
$$\AAA_{n} = \left\{  \mathcal{A} \in \mathbb{R}^{n}:\; 0< a_{i} \leq 1, \; \sum_{i} a_{i} >2 \right\} \subset \mathbb{R}^{n}.$$

The hyperplanes $\sum_{i \in I} a_{i} = 1$, $I \subset [n], |I| \geq 2$, (called \textit{walls}) cut the polytope $\AAA_{n}$ into \textit{chambers}. The Hassett compactification depends on a chamber only \cite[Proposition 5.1]{Has03}.

Combinatorial stratification of the space $\MMM_{0, \mathcal{A}}$ looks similarly to that of the Deligne--Mumford's with the only difference --- some of the marked points now can coincide \citep{Cey07}. 

More precisely, a \textit{weighted tree} $(\gamma, I)$ is an ordered $k$-partition $I_{1}\sqcup \dots \sqcup I_{k} = [n]$ and a tree $\gamma$ with $k$ ordered leaves marked by elements of the partition such that (1) $\sum_{j \in I_{m}} a_{j} \leq 1$ for any $m$, and (2) for each vertex, the number of emanating bold edges plus the total weight is greater than $2$. Open strata are enumerated by weighted trees:
the stratum of the space $\MMM_{0,\mathcal{A}}$ corresponding to a weighted tree $({\gamma}, I)$ consists of curves whose irreducible components form the tree $\gamma$ and collisions of sections form the partition $I$. Closure of this stratum is denoted by $\D_{({\gamma}, I)}$.

\subsection{Polygon spaces as compactifications of $\MM_{0,n}$}\label{FlePolcom}

Assume that an $n$-tuple of positive real numbers $\mathcal{L} = (l_{1},...,l_{n})$ is fixed. We associate with it a \textit{flexible polygon}, that is, $n$  rigid bars of lengths $l_{i}$ connected in a cyclic chain by revolving joints. A \textit{configuration} of $\mathcal{L}$ is an $n$-tuple  of points $(q_{1},...,q_{n}), \; q_i \in \mathbb{R}^3,$ with $|q_i q_{i+1}|=l_{i}, \; \; |q_{n} q_{1}|=l_{n}$.

The following two  definitions for the \textit{polygon space}, or the \textit{moduli space of the flexible polygon} are equivalent:

\begin{definition} \cite{HKn}
\begin{itemize}
\item[I.]The moduli space $M_{\mathcal{L}}$ is a set of all configurations of $\mathcal{L}$  modulo orientation preserving  isometries of $\mathbb{R}^3$.

\item[II.] Alternatively, the space $M_{ \mathcal{L}}$ equals the quotient of the space $$\left\{(u_1,...,u_n) \in (\mathbb{S}^2)^n : \sum_{i=1}^n l_iu_i=0\right\}$$ by the diagonal action of the group $\SO_3(\mathbb{R})$.
\end{itemize}
\end{definition}

The second definition shows that the space  $M_{\mathcal{L}}$ does not depend on the
ordering of $\{l_1,...,l_n\}$; however, it does depend on the values
of $l_i$. 

%{\Huge UP to here is done!}

Let us consider the \textit{parameter space} $$\left\{ (l_{1}, \dots, l_{n}) \in \mathbb{R}^{n}_{>0}:\ l_{i}< \sum_{j\neq i} l_{j} \text{ for }i=1, \dots, n\right\}.$$

This space is cut into open \textit{chambers} by \textit{walls}. The latter are hyperplanes with defining equations
$$\sum_{i\in I} l_i = \sum_{j\notin I} l_j.$$

The diffeomorphic type of $M_{ \mathcal{L}}$ depends only on the chamber containing $\mathcal{L}$. For a point $\mathcal{L}$ lying strictly inside some chamber, the space $M_{\mathcal{L}}$ is a smooth $(2n-6)$-dimensional  algebraic variety \cite{Kl}. In this case we say that the length vector is \textit{generic}. 

\begin{definition}
For a generic length vector $\mathcal{L}$, we call a subset $J\subset [n]$ {\it long} if $$\sum_{i \in J} l_i > \sum_{i\notin J} l_i.$$
Otherwise, $J$ is called \textit{short}. The set of all short sets we denote by $SHORT(\mathcal{L})$.
\end{definition}

Each subset of a short set is also short, therefore $SHORT(\mathcal{L})$ is  a (threshold Alexander self-dual) simplicial complex.  Rephrasing the above, the diffeomorphic type of $M_{ \mathcal{L}}$ is defined by the simplicial complex $SHORT(\mathcal{L})$.

\section{ASD and preASD compactifications}\label{ASDpreASD}

\subsection{ASD and preASD simplicial complexes}
Simplicial complexes provide a necessary combinatorial framework  for the description of the category of smooth modular compactifications of $\MM_{0,n}$.

A \textit{simplicial complex} (a \textit{complex,} for short) $K$ is a subset of $2^{[n]}$ with the hereditary property: $A \subset B\in K$ implies $A \in K$. Elements of $K$ are called \textit{faces} of the complex. Elements of $2^{[n]} \setminus K$ are called \textit{non-faces}. The maximal (by inclusion) faces  are called \textit{facets}.

We assume that the set of $0$-faces (the set of vertices) of a complex is $[n]$. The complex $2^{[n]}$ is denoted by $\Delta_{n-1}$. Its $k$-skeleton is denoted by $\Delta_{n-1}^{k}$. {In particular, $\Delta_{n-1}^{n-2}$ is the boundary complex of the simplex $\Delta_{n-1}$.} 

\begin{definition}
For a  complex $K \subset 2^{[n]}$,  its \textit{Alexander dual}  is  the simplicial complex 
$$K^{\circ}:= \{A \subset [n]:\; A^{c} \not\in K\} = \{A^{c}:\; A \in 2^{[n]}\backslash K\}.$$
Here and in the sequel,  $A^c=[n]\setminus A$ is the complement of $A$.
A  complex $K$ is \textit{Alexander self-dual (an ASD complex)} if $K^{\circ}=K$. A \textit{pre Alexander self-dual (a pre ASD)} complex is a complex contained in some ASD  complex.
\end{definition}

In other words, ASD complexes (pre ASD complexes, respectively)  are characterized by the condition:
 for any partition $[n]=A\sqcup B$, {exactly one} (at most one, respectively) of $A$, $B$ is a face.

\medskip

Some ASD complexes are \textit{threshold complexes}: they equal $SHORT(\mathcal{L})$ for some generic weight vectors $\mathcal{L}$ (Section \ref{FlePolcom}).  It is known that threshold ASD complexes exhaust all ASD complexes for $n \leq 5$. However, for bigger $n$ this is no longer true. Moreover, for $n \rightarrow \infty$ the percentage of threshold ASD complexes tends to zero.

To produce new examples of ASD complexes, we use \textit{flips}:
\begin{definition}\cite{PanGal}
For an ASD complex $K$ and a facet $A \in K$ we build a new ASD complex
$$\flip_{A}(K):= (K\backslash A) \cup A^{c}.$$
\end{definition}

It is easy to see  that
\begin{proposition}\label{LemmaNotLonger2}
(1) \cite{PanGal} Inverse of a flip is also some flip.
(2) \cite{PanGal} Any two ASD complexes are connected by a sequence of flips. (3) For any ASD complex $K$ there exists a threshold ASD complex $K'$ that can be obtained  from $K$ by a sequence of flips with some $A_{i}\subset [n]$ such that  $|A_{i}| > 2, |A_{i}^c| > 2$.

\end{proposition}

\begin{proof} We prove (3).
It is sufficient to show that for any ASD complex, there exists a threshold ASD complex with the same collection of $2$-element non-faces. For this, let us observe that any two 
non-faces of an ASD complex necessarily intersect. Therefore, all possible collections of $2$-element non-faces of an ASD complex  (up to renumbering) are:
\begin{enumerate}
\item empty set;
\item $(12),\  (23),\ (31)$;
\item $(12),(13),\dots ,(1k)$.
\end{enumerate}
It is easy to find appropriate threshold ASD complexes for all these cases.
\end{proof}

ASD complexes appear in the game theory literature   as ``simple games with constant
sum'' (see \citep{vNMor}). One imagines $n$ players and all possible ways of partitioning them into two teams. The teams compete, and  a team \textit{looses} if it belongs to $K$. In the language of flexible polygons, a short set is a loosing team.

\medskip
\textbf{Contraction, or freezing operation.}  Given an ASD complex $K$, let us build a new ASD complex  $K_{(ij)}$  with $n-1$ vertices
$\{1,...,\widehat{{i}},...,\widehat{{j}},...,n,(i,j)\}$ by \textit{contracting the edge} $\{i,j\}\in K$, or \textit{freezing}   $i$ and $j$ together.

The formal definition is: for $A\subset \{1,...,\widehat{{i}},...,\widehat{{j}},...,n\}$, $A\in K_{(ij)}$ iff $A\in K$, and 
 $A\cup \{(ij)\}\in K_{(ij)}$ iff $A \cup \{i,j\}\in K$.
 
 Contraction  $K_I$ of any other face $I\in K$ is defined analogously.
 
Informally, in the language of simple game, contraction of an edge means making one player out of two.
In the language of flexible polygons, ``freezing'' means producing one new edge  out of two old ones (the lengths sum up).

\begin{figure}
\includegraphics[scale=1]{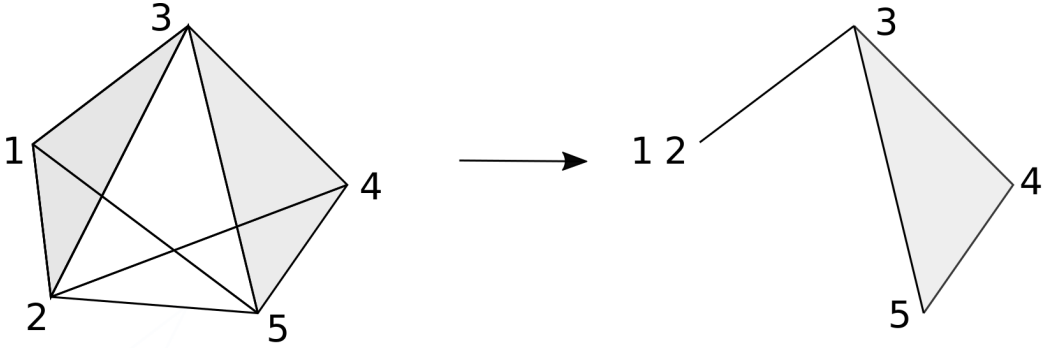}
\caption{Contraction of $\{1,2\}$ in a simplicial complex}
\label{Figcont}
\centering
\end{figure}

\subsection{Smooth extremal assignment compactifications}\label{Smycom}

Now we review the results of  \cite{Smyth13} and \cite{Han15}, and indicate a relation with preASD complexes.

For a scheme $B$, consider the space $\mathcal{U}_{\,0, n}(B)$ of all flat, proper, finitely-presented morphisms $\pi: \;\mathcal{C} \rightarrow B$ with $n$ sections $\{s_{i}\}_{i \in [n]}$, and connected, reduced, one-dimensional geometric fibers of genus zero.  Denote  by $\mathcal{V}_{0, n}$ the irreducible component of $\mathcal{U}_{\,0,n}$ that contains $\MM_{0,n}$.
\begin{definition}
A modular compactification of $\MM_{0,n}$ is an open substack $\mathcal{X} \subset \mathcal{V}_{0,n}$ that is proper over $\mathbb{Z}$. A modular compactification is stable  if every geometric point $(\pi:\;\mathcal{C} \rightarrow B; s_{1}, \dots, s_{n})$ is stable. We call a modular compactification {smooth} if it is a smooth algebraic variety.
\end{definition}

\begin{definition} A smooth extremal assignment $\mathcal{Z}$ over $\MMM_{0,n}$ is an assignment to each stable tree with $n$ leaves a subset of its vertices
$$\gamma \mapsto \mathcal{Z}(\gamma) \subset Vert(\gamma)$$
such that:
\begin{enumerate}
\item for any tree $\gamma$, the assignment is a proper subset of vertices: $\mathcal{Z}(\gamma) \subsetneqq Vert(\gamma)$,
%\item for any graph $\gamma$ assignment $\mathcal{Z}(\gamma)$ is invariant under the automorphism group $Aut(\gamma)$ action\footnote{In the considering case of zero genus, this condition is trivial.},\textbf{REMOVE THIS?}
\item for any contraction $\gamma \rightsquigarrow \tau$ with $\{v_{i}\}_{i \in I} \subset Vert(\gamma)$ contracted to $v \in Vert(\tau)$, we have $v_{i}\in \mathcal{Z}(\gamma)$ for all $i \in I$ if and only if $v \in \mathcal{Z}(\tau)$.
\item for any tree $\gamma$ and $v \in \mathcal{Z}(\gamma)$ there exists a two-vertex tree $\gamma'$ and $v'\in \mathcal{Z}(\gamma')$ such that $$\gamma' \rightsquigarrow \gamma \text{ and }v' \rightsquigarrow v.$$\end{enumerate}\end{definition}
\begin{definition}\label{DefAss}
Assume that $\mathcal{Z}$ is a smooth extremal assignment. A curve $(\pi:\;\mathcal{C} \rightarrow B; s_{1}, \dots, s_{n})$ is $\mathcal{Z}$--stable if it can be obtained from some Deligne--Mumford stable curve $(\pi':\;\mathcal{C}' \rightarrow B'; s_{1}', \dots, s_{n}')$ by (maximal) blowing down irreducible components of the curve $\mathcal{C}'$ corresponding to the vertices from the set $\mathcal{Z}(\gamma(\mathcal{C}'))$.
\end{definition} 

A smooth assignment is completely defined by its value on two-vertex stable trees with $n$ leaves. The latter bijectively correspond to unordered partitions $A\sqcup A^{c} = [n]$ with $|A|, |A^c| > 1$: sets $A$ and $A^{c}$ are affixed to two vertices of the tree. The first condition  of Definition \ref{DefAss} implies that no more than one of $A$ and $A^{c}$ is ``assigned''. One concludes that  preASD complexes are in bijection with smooth assignments.

All possible modular compactifications of $\MM_{0, n}$ are parametrized by smooth extremal assignments:

\begin{theorem}\cite[Theorems 1.9 \& 1.21]{Smyth13} and \cite[Theorem 1.3]{Han15}
\begin{itemize}
\item For any smooth extremal assignment $\mathcal{Z}$ of $\MM_{0,n}$, or equivalently, for any preASD complex $K$, there exists a stack $\MMM_{0,\mathcal{Z}} = \MMM_{0,K} \subset \mathcal{V}_{0,n}$ parameterizing all $\mathcal{Z}$--stable rational curves.
\item For any smooth modular compactification $\mathcal{X} \subset \mathcal{V}_{0,n}$, there exist a smooth extremal assignment $\mathcal{Z}$ (a preASD complex $K$) such that $\mathcal{X} = \MMM_{0,\mathcal{Z}} = \MMM_{0,K}$.
\end{itemize}
\end{theorem}

There are two different ways to look at a moduli spaces. In the present paper we  look at the moduli space as at a smooth algebraic variety equipped with $n$ sections (\textit{fine moduli space}). The other way is to look at it as at a smooth algebraic variety (\textit{coarse moduli space}). Different preASD complexes give rise to different fine moduli spaces. However, two different complexes can yield isomorphic coarse moduli spaces.

Indeed, consider two preASD complexes $K$ and $K\cup\{ij\}$ (we abbreviate the latter as $K+(ij)$), assuming that $\{ij\}\notin K$. The corresponding algebraic varieties $\MMM_{0, K}$ and $\MMM_{0, K + (ij)}$ are isomorphic. A vivid explanation is: to let a couple of marked points to collide is the same as to add a nodal curve with these two points sitting alone on an irreducible component. Indeed, this irreducible component would have exactly three special points, and $\PSL_{2}$ acts transitively on triples.

\begin{theorem}\cite[Statements 7.6--7.10]{Han15}\label{objects_fixedlevel}
The set of smooth modular compactifications of $\MM_{0,n}$ is in a bijection with objects of the ${\bf preASD}_{n}/ \thicksim$, where $K\thicksim L$ whenever $K\setminus L$ and $L\setminus K$ consist of two-element sets only.
\end{theorem}

\begin{example}\label{example}
PreASD  complexes and corresponding compactifications.
\begin{enumerate}
\item the 0-skeleton $\Delta_{n-1}^{0} = [n]$ of the  simplex $\Delta_{n-1}$ corresponds to the Deligne--Mumford compactification; 
\item {the complex $\mathcal{P}_{n} : = {\bf pt} \sqcup \Delta_{n-2}^{n-3}$ (disjoint union of a point and the boundary of a simplex $\Delta_{n-2}$)} is ASD. It corresponds to the Hassett weights $(1, \varepsilon, \dots, \varepsilon)$; this compactification is isomorphic to $\mathbb{P}^{n-3}$;
\item the Losev--Manin compactification $\MMM_{0, n}^{LM}$ \cite{LosMan00} corresponds to the weights $(1,1,\varepsilon, \dots, \varepsilon)$ and to the complex ${\bf pt}_{1} \sqcup {\bf pt}_{2} \sqcup \Delta_{n-3}$;
\item the space $(\mathbb{P}^{1})^{n-3}$ corresponds to weights $(1,1,1,\varepsilon, \dots, \varepsilon)$, and to the complex ${\bf pt}_{1} \sqcup {\bf pt}_{2} \sqcup {\bf pt}_{3} \sqcup \Delta_{n-4}$.
\end{enumerate}
\end{example}

\medskip

\subsection{ASD compactifications
via  stable point configurations}
 
ASD compactifications can be explained  in a self-contained way, without referring to \cite{Smyth13}.

Fix an ASD complex $K$ and consider configurations of $n$ (not necessarily all distinct) points $p_1,...,p_n$ in the projective line. A configuration is called  {\em stable} if the index set of each collection of coinciding points belongs to $K$. That is, whenever $p_{i_1}=...=p_{i_k}$, we have $\{i_1,...,i_k\}\in K$.

Denote by  $STABLE(K)$ the space of stable configurations in the  complex projective line.
The group  $\mathrm{PSL}_{2}(\mathbb{C})$  acts naturally on this space. Set $$\MMM_{0, K}:=STABLE(K)/\mathrm{PSL}_{2}(\mathbb{C}).$$

If $K$ is a threshold complex, that is, arises from some flexible polygon $\mathcal{L}$, then  the space $\MMM_{0, K}$ is isomorphic to the polygon space $M_{\mathcal{L}}$  \cite{Kl}.

\bigskip
Although the next theorem fits in a broader context of \cite{Smyth13}, we give here its elementary proof.
\begin{theorem}\label{ThmSmoothComp}
The space $\MMM_{0, K}$ is a  compact smooth variety with a natural complex structure.
\end{theorem}
\begin{proof}$\;$
\textbf{Smoothness.} For a distinct triple of indices $i,j,k \in [n]$, denote by $U_{i,j,k}$ the subset of $\MMM_{0, K}$ defined by  $p_i\neq p_j,$ $p_j \neq p_k$, and $p_i \neq p_k$.   For each of $U_{i,j,k}$, we get rid of the action of the group $\mathrm{PSL}_{2}(\mathbb{C})$, setting
$$U_{i,j,k}=\big\{(p_1,...,p_n)\in \MMM_{0, K}:\; p_i=0,p_j=1,\text{ and }p_k=\infty  \big\}.$$

Clearly, each of the charts  $U_{i,j,k}$ is an open smooth manifold. Since all the $U_{i,j,k}$ cover $\MMM_{0, K}$, smoothness is proven.

\medskip

\textbf{Compactness.} Let us show that each sequence of $n$-tuples has a converging subsequence.

Assume the contrary. Without loss of generality, we may assume that the sequence  $(p_1^i=0,p_2^i=1,p_3^i=\infty,p_4^i,...,p_n^i)_{i=1}^{\infty}$ has no converging subsequence. We may assume that for some set $I \notin K$,  all  $p^i_j$  with $j\in I$ converge to a common point. We say that we have a \textit{collapsing long set} $I$. This notion depends on the choice of a chart. We may assume that our collapsing long set has the minimal cardinality among all long sets that can collapse without a limit (that is, violate compactness) for this complex $K$. We may assume that  $I=\{3,4,5,...,k\}$. 

This long set can contain at most one of the points $p_1,p_2,p_3$. We consider the case when it contains $p_3$; other cases are treated similarly.

That is, all the points $p^i_4,...,p^i_k$ tend to $\infty$. Denote by $C_i$ the circle with the minimal radius embracing the points $p^i_3=\infty ,p^i_4,p^i_5,...,p^i_k$. The circle contains at least two points of $p^i_4,...,p^i_k, p_3=\infty$. Apply a transform $\phi_i \in \mathrm{PSL}_{2}(\mathbb{C})$ which turns the radius of $C_i$ to $1$, and keeps at least two of the points $p^i_4,...,p^i_k, p_3=\infty$ away from each other. In this new chart the cardinality of the collapsing long set gets smaller. A contradiction to the minimality assumption.
\end{proof}

A natural question is: what if one  takes a simplicial complex  (not a self-dual one), and cooks the analogous quotient space. Some heuristics are: if the complex contains simultaneously  some  set $A$  and its complement $[n]\setminus A$, we  have a stable tuple with a non-trivial stabilizer in $\mathrm{PSL}_{2}(\mathbb{C})$, so the factor has a natural nontrivial orbifold structure. If a simplicial complex is smaller than some ASD complex $K'$, and therefore, we get a proper open subset of $\MMM_{0, K'}$, that is, we lose compactness.

\subsection{Perfect cycles}
Assume that we have an ASD complex $K$ and the associated compactification
$\MMM_{0, K}$. Let $K_I$ be the contraction of some face $I\in K$. Since the variety $\MMM_{0, K_I}$ naturally embeds in $ \MMM_{0, K}$, the contraction procedure gives rise to a number of subvarieties of $\MMM_{0, K}$. These varieties (1) ``lie on the boundary'' \footnote{That is do not intersect the initial space $\MM_{0, n}$.} and (2)  generate the Chow ring (Theorem \ref{ASDChow} ). Let us look at them in more detail.

\medskip

An \textit{elementary perfect cycle} $(ij)=(ij)_{K}\subset \MMM_{0, K}$  is defined as

$$(ij)=(ij)_K=\{(p_1,...,p_n)\in \MMM_{0, K}: p_i=p_j\}.$$

Let $[n]=A_1\sqcup...\sqcup A_k$ be an unordered partition of $[n]$. A \textit{perfect cycle} associated to the partition 
\begin{align*}
(A_1)\cdot...\cdot(A_k)&=(A_1)_{K}\cdot...\cdot(A_k)_{K} = \\
&= \{(p_1,...,p_n)\in \MMM_{0, K}: i,j \in A_m \Rightarrow p_i=p_j\}.
\end{align*}

%\textcolor{red}{From the codimension count we can see that the cycle $(A)$ is a complete intersection of the elementary perfect cycles enumerated by any tree with vertex set equals $A$. For example, the cycle $(1234)$ is a complete intersection of the divisors $(12), (23),$ and $(14)$ or, equivalently, of the divisors $(12), (23),$ and $(34)$.} 

Each perfect cycle is isomorphic to $\MMM_{0, K'}$ for some complex $K'$ obtained from $K$ by a series of contractions.

Singletons play no role, so we omit all one-element sets $A_i$ from our notation. Consequently, all the perfect cycles are labeled by partitions of some subset of $[n]$ such that all the $A_i$ have at least two elements.

Note that for arbitrary $A \in K$, the complex $K_A$ might be ill-defined. This happens if $A \notin K$.
In this case the associated perfect cycle $(A)$ is empty.

 For each perfect cycle there is an associated Poincaré dual element in the cohomology ring. These dual elements we denote by the same symbols as the perfect cycles. 

\medskip

The following rules allow to compute the cup-product of perfect cycles: 

\begin{proposition}\label{ComputRules}$\;$
\begin{enumerate}
                                      \item  Let $A$ and $B$ be disjoint subsets of $[n]$. Then
                                      \begin{enumerate}

  \item $(A)\smile (B)=(A)\cdot (B).$

  \item $(Ai)\smile  (Bi)=(ABi).$

\end{enumerate}
                                      \item For $A\notin K$, we have $(A)=0$. If one of $A_k$ is a non-face of $K$, then $(A_1)\cdot...\cdot(A_k)=0$.

                                      \item  The four-term relation:

$(ij)+(kl)=(jk)+(il)$ holds for any distinct $i,j,k,l$.
                               
\end{enumerate}
\end{proposition}

\begin{proof}
In the cases (1) and (2)  we have a transversal intersection of holomorphically embedded complex varieties. The item (3) will be proven in Theorem \ref{ASDChow}.
\end{proof}

Examples:
$$(123)\cdot (345)=(12345);\ \ (12)\cdot (34) \cdot (23)=(1234).$$
A more sophisticated computation:
$$(12)\cdot(12) = (12)\cdot \big( (13) + (24) - (34) \big) = (123) + (124) -(12)\cdot(34).$$

\begin{proposition}\label{PropPerf}A cup product of  perfect cycles is a perfect cycle.

\end{proposition}
\textit{Proof.}  Clearly, each perfect cycle is a product of elementary ones.
Let us prove that the product of two perfect cycles is an integer linear combination of   perfect cycles.
We may assume that the second factor is an elementary perfect cycle, say, $(12)$. Let the first factor be $(A_1)\cdot(A_2)\cdot(A_3)\cdot\ldots\cdot(A_k)$.

We need the following case analysis:
\begin{enumerate}
\item If at least one of $1,2$ does not belong to $\bigcup A_i$, the product is a perfect cycle by  Proposition \ref{ComputRules}, (1).
\item If $1$ and $2$ belong to different $A_i$, we use the following:

for any perfect cycle $(A_1)\cdot(A_2)$ with $ i\in A_1, j\in A_2$, we have
$$(A_1)\cdot(A_2)\smile (ij)=(A_1)\smile (ij)\smile (A_2)= (A_1j)\smile (A_2)= (A_1\cup A_2).$$

\item Finally, assume that $1,2 \in A_1$. Choose $i\notin A_1,\ \ j\notin A_1$ such that $i$ and $j$ do not belong to one and the same $A_k$. By Proposition \ref{ComputRules}, (3), $$(A_1)\cdot(A_2)\cdot(A_3)\cdot\dotso\cdot(A_k)\smile (12)=(A_1)\cdot(A_2)\cdot(A_3)\cdot\ldots\cdot(A_k)\smile \big((1i)+(2j)-(ij)\big).$$
After expanding the brackets, one reduces this to the above cases.\qed
\end{enumerate}
                                    
\begin{lemma}\label{LemmaTriple}
For an ASD complex $K$, let $A\sqcup B\sqcup C=[n]$ be a partition of $[n]$ into three faces. Then $(A)\cdot(B)\cdot(C)=1$ in the  graded component ${\bf A}^{n-3}_K$ of the Chow ring, which is canonically identified with $\mathbb{Z}$.
\end{lemma}
\begin{proof}  
Indeed, the cycles $(A)$, $(B)$, and $(C)$ intersect transversally at a unique point.
\end{proof} 

Now we see that the set of perfect cycles is closed under cup-product.  In the next section we  show that the Chow ring equals the ring of perfect cycles.

\medskip
\subsection{Flips and blow ups. }

Let  $K$ be an  ASD complex, and let $A \subset [n]$ be its facet. 

\begin{lemma}\label{LemmaPerfCyclProj}
The perfect cycle $(A)$ is isomorphic to $\MMM_{0, \mathcal{P}_{|A^c|+1}} \cong \mathbb{P}^{|A^c|-2}$.
\end{lemma}
\begin{proof} Contraction of $A$ gives the complex ${\bf pt}\;\sqcup\;\Delta_{|A^c|}^{} = \mathcal{P}_{|A^c|+1}$ from the Example \ref{example}, (2). \end{proof}

\begin{lemma}\label{blowupVSflip}
For  an ASD complex $K$  and its facet $A$, there are two blow up morphisms 
$$\pi_{ A }: \MMM_{0,K\backslash A} \rightarrow \MMM_{0, K} \text{ and }\pi_{A^c}: \MMM_{0,K\backslash A} \rightarrow \MMM_{0, \flip_{A} (K)}.$$
The centers of these blow ups are the perfect cycles $(A)$ and $(A^c)$ respectively.  The exceptional divisors are equal: $\D_{A} = \D_{A^c}$. Both are     isomorphic to $\MMM_{0, \mathcal{P}_{|A|+1}} \times \MMM_{0, \mathcal{P}_{|A^c|+1}} \cong \mathbb{P}^{|A|} \times \mathbb{P}^{|A^c|}$. The maps $\pi_{A}|_{\D_{A}}$ and $\pi_{A^c}|_{\D_{A^c}}$ are projections to the first and the second components respectively.
\end{lemma}

The \textit{proof} literally repeats  \cite[Corollary 3.5]{Has03}: $K$--stable but not $K_{A}$($K_{A^c}$)--stable  curves have two connected components. The  marked points with indices from the set $A$  lie on one of the irreducible components, and  marked points with indices from the set $A^c$ lie on the other. \qed

\begin{corollary}
For an ASD complex $K$ and its  facet $A$, the algebraic varieties $\MMM_{0, K}$ and $\MMM_{0, K \backslash A}$ are HI--schemes, i.e., the canonical map from the Chow ring to the cohomology ring is an isomorphism.
\end{corollary}
\textit{Proof.} This follows from Lemma \ref{blowupVSflip}  and Theorem~\ref{ChowCoh}.\qed

\section{Chow rings of ASD compactifications}

As it was already mentioned, many examples of ASD compactifications  are polygon spaces, that is, come from a threshold ASD complex. Their Chow rings  were computed in \cite{HKn}. A more relevant to the present paper presentation of the ring is given in  \cite{NP}. We recall it below.

\begin{definition}
Let ${\bf A}^{*}_{univ} = {\bf A}^{*}_{univ, n}$ be the ring 
$$\mathbb{Z}\big[(I):\; I \subset [n], 2 \leq |I| \leq n-2 \big]$$
factorized by  relations:
\begin{enumerate}
\item ``\textit{The four-term relations}'': $(ij)+(kl)-(ik)-(jl)=0$ for any $i,j,k,l \in [n]$.

\item ``\textit{The multiplication rule}'': $(Ik)\cdot(Jk) = (IJk)$ for any disjoint $I, J \subset [n]$ not containing element $k$.

\end{enumerate}
\end{definition}

There is a natural graded ring homomorphism from ${\bf A}^*_{univ}$  to the Chow ring of an ASD-compactification that sends each of the generators $(I)$ to the corresponding perfect cycle.

\begin{theorem}\cite{NP}\label{ChowForPolygons}
The Chow ring (it equals the cohomology ring) of a polygon space equals the ring  ${\bf A}^*_{univ}$  factorized by 
$$(I)=0 \ \ \ \hbox{whenever $I$ is a long set.}$$
\end{theorem}

The following generalization of Theorem \ref{ChowForPolygons} is the first main result of the paper:

\begin{theorem}\label{ASDChow}
For an ASD complex $L$, the Chow ring ${\bf A}^{*}_{L}:= {\bf A}^{*}(\MMM_{0, L})$ of the moduli space $\MMM_{0, L}$ is isomorphic to the quotient  ${\bf A}^{*}_{univ}$ by the ideal $\mathcal{I}_{L}:= \big\langle (I):\;I\not\in L \big\rangle$.

\end{theorem}

The  idea of the proof is: the claim is true for threshold ASD complexes (i.e., for polygon spaces), and each ASD complex is achievable from a threshold ASD complex  by  a sequence of flips. Therefore it is sufficient to look at a unique flip.
Let us consider an ASD complex $K + B$ where $B\notin K$ is a  facet in $K + B$.  Set $A:=[n]\setminus B$, and consider the ASD complex  $K + A= \flip_{B}(K+B)$.
\begin{center}
\begin{tikzcd}[column sep=small]
& K \arrow[dl, hook] \arrow[dr, hook] & \\
K+ B \arrow[rr, dashrightarrow,  "\flip_{B}"] & & K+A
\end{tikzcd}
\end{center}

We are going to prove that if the claim of the theorem  holds true for $K+B$, then it also holds for $K+A$.

\medskip

 By Lemma \ref{blowupVSflip},
 the space $\MMM_{0, K}$ is the blow up of $\MMM_{0, K+B}$ along the subvariety $(B)$ and the blow up of $\MMM_{0, K+A}$ along the subvariety $(A)$. The diagram of the blow ups looks as follows:
\begin{center}
\begin{tikzcd}
(B)\arrow[d, hook, "i_{B}"] & \D \arrow[hook]{d}{j_{A} = j_{B}} \arrow{r}{g_{A}}\arrow{l}{g_{B}} & (A) \arrow[hook]{d}{i_{A}} \\
 \MMM_{0, K+B} & \MMM_{0, K} \arrow[twoheadrightarrow]{r}{\pi_{A}}\arrow[l, two heads, "\pi_{B}"]  & \MMM_{0, K+A }
\end{tikzcd}
\end{center}
The induced diagram of Chow rings is:
\begin{center}
\begin{tikzcd}
{\bf A}^{*}_{(B)} = {\bf A}^{*}_{\mathcal{P}_{|A|+1}} \arrow{r}{g_{B}^{*}} & {\bf A}^{*}_{\mathcal{P}_{|A|+1}}\times {\bf A}^{*}_{\mathcal{P}_{|B|+1}}  & {\bf A}^{*}_{(A)} = {\bf A}^{*}_{\mathcal{P}_{|B|+1}}  \arrow{l}{g_{A}^{*}} \\
 {\bf A}^{*}_{K + B}\arrow[u, "i_{B}^{*}"]\arrow[r, hookrightarrow, "\pi_{B}^{*}" description] & {\bf A}^{*}_{K} \arrow[hook]{u}{j_{A}^{*} = j_{B}^{*}}  & {\bf A}^{*}_{K + A} \arrow[l, hookrightarrow ,  "\pi_{A}^{*}" description]\arrow[u, "i_{A}^{*}" description]
\end{tikzcd}
\end{center}

Let  ${\bf A}^{*}_{K+ A, comb}$  be the quotient of ${\bf A}^{*}_{univ}$ by the ideal $\mathcal{I}_{K+ A}$.
We have  a natural graded ring homomorphism  $$\alpha = \alpha_{K+A}:{\bf A}^{*}_{K+ A, comb} \rightarrow {\bf A}^{*}_{K+ A}=:{\bf A}^{*}_{K+ A, alg},$$  where the map $\alpha$ sends each symbol $(I)$ to the associated perfect cycle.

A remark on notation: as a general rule, all objects related to ${\bf A}^{*}_{K+ A, comb}$ we mark with a subscript \textit{``comb''}, and objects related to ${\bf A}^{*}_{K+ A, alg}$  we mark with \textit{``alg''}. 
\medskip
\medskip

We shall show that $\alpha$ is an isomorphism.
The outline of the proof is:

\begin{enumerate}
\item  The ring ${\bf A}^{*}_{K+A,alg}$ is generated by the first graded component. (The ring ${\bf A}^{*}_{K+A,comb}$ is also generated by the first graded component; this is clear by construction.) 
\item The restriction of $\alpha$ to the first graded components 
is a group isomorphism. Therefore, $\alpha$ is surjective.
\item The map $\alpha$ is injective.
\end{enumerate}

\medskip

\iffalse We make use of the general structure theory for Chow rings under blow ups (the main theorems we collect in Appendix).
\fi

%\bigskip
\begin{lemma} \label{OneGen} The ring ${\bf A}^{*}_{K+A, alg}$   is generated by the group ${\bf A}^{1}_{K+A, alg}$. \end{lemma} 

\begin{proof}
By Theorem \ref{NewChowRing}
$${\bf A}^{*}_{K} \cong \frac{ {\bf A}^{*}_{K+A, alg}[T]}{\big( f_{A}(T), T\cdot \ker (i_{A}^{*})\big)}.$$

Observe that:

\begin{itemize}
\item The zero graded components of ${\bf A}^{*}_{K+A, alg}, {\bf A}^{*}_{K+A, comb}$ equals $\mathbb{Z}$.
\item The  map $\pi_{A}^*:\;{\bf A}^{*}_{K+A, alg} \rightarrow {\bf A}^{*}_{K}$ is a homomorphism of graded rings. Moreover, the variable $T$ stands for the additive inverse of the class of the exceptional divisor $\D$. And so, $T$ a degree one homogeneous element.

\item  Since $i_{A}^*$ is the multiplication by the cycle $(A)$, the kernel $\ker(i_{A}^{*})$ equals the annihilator $\Ann(A)_{alg}$ in the ring ${\bf A}^{*}_{K+A, alg}$. {Since the space $\MMM_{0, K+A}$ is an HI-scheme}, the degree of the ideal $\Ann(A)_{alg}$ is strictly positive.

\item The  polynomial $f_{A}(T)$ is  a homogeneous element whose degree  equals the degree $\deg_{T}(f_{A}(T))$. Besides, its coefficients are generated by elements from the first graded component since they all belong  to the ring $\alpha ({\bf A}^{*}_{K+A, comb})$.
\end{itemize}

Denote by $\langle{\bf A}^1_{K+A, alg}\rangle$ the subalgebra of ${\bf A}_{K+A, alg}$ generated by the first graded component.

First observe that the restriction of the map ${\bf A}^{*}_{K+A, alg}[T] \rightarrow {\bf A}^{*}_{K}$ to the first graded components is injective. 

Assuming that the lemma is not true,
 consider a homogeneous element $r$ of the ring ${\bf A}^{*}_{K+A, alg}$ with minimal degree among all not belonging to  $\langle{\bf A}^{1}_{K+A, alg}\rangle$.
There exist elements $b_{i}\in  \langle{\bf A}^{1}_{K+A, alg}\rangle$   such that $b_{p}\cdot T^{p} + \dots + b_{1}\cdot T + b_{0} = r$ in the ring ${\bf A}^{*}_{K+A, alg}[T]$. The elements   $b_{i}$ are necessarily homogeneous.

Equivalently, $b_{p}\cdot T^{p} + \dots + b_{1}\cdot T + b_{0} - r$ belongs to the ideal $\big(f_A(T), T\cdot \Ann(A_{alg})\big)$.  Therefore $b_{p}\cdot T^{p} + \dots + b_{1}\cdot T + b_{0} - r=x\cdot f_A(T) + y \cdot T \cdot i$ with  some $x, y \in R[T]$ and $i \in \Ann(A_{alg})$.

Setting $T=0$, we get  $b_{0} - r = x_{0}\cdot f_{0}$. If the element $x_{0}$ belongs to  $\langle{\bf A}^{1}_{K+A, alg}\rangle$, then we are done. Assume  the contrary. Then from the minimality assumption  we get the following inequalities: $\deg (b_0 - r) = \deg (x_{0}\cdot f_{0}) > \deg(x_{0}) \geq \deg(r)$. A contradiction.
\end{proof}

\begin{lemma} \label{FirstGroup}
For any ASD complex $L$ the groups ${\bf A}^1_{L, comb}$ and ${\bf A}^1_{L, alg}$ are isomorphic. The isomorphism is induced by the homomorphism $\alpha_{L}$.
\end{lemma}
The \textit{proof} analyses how do these groups change under flips.

We know that the claim is true for threshold complexes. Due to Lemma \ref{LemmaNotLonger2} we  may consider flips only  with $n-2>|A|>2$.
Again, we suppose that the claim is true for the complex $K+B$ and will prove for the complex $K+A$ with $A \sqcup B = [n]$.
Under such flips ${\bf A}^{1}_{comb}$  does not change.  The group ${\bf A}^{1}$  does not change neither. This becomes clear with the following two short exact sequences (see Theorem \ref{SESBlow},e):
\begin{align*}
0 &\rightarrow {\bf A}_{n-4}\big(\MMM_{0, \mathcal{P}_{|A|+1}}\big) \rightarrow {\bf A}_{n-4}\big( \MMM_{0, \mathcal{P}_{|A|+1} } \times \MMM_{0, \mathcal{P}_{|B|+1} } \big) \oplus {\bf A}_{n-4}\big( \MMM_{0, K + B} \big) \rightarrow {\bf A}_{n-4}\big(\MMM_{0, K}\big) \rightarrow 0,\\
0 &\rightarrow {\bf A}_{n-4}\big(\MMM_{0, \mathcal{P}_{|B|+1}}\big) \rightarrow {\bf A}_{n-4}\big( \MMM_{0, \mathcal{P}_{|A|+1} } \times \MMM_{0, \mathcal{P}_{|B|+1} } \big) \oplus {\bf A}_{n-4}\big( \MMM_{0, K + A} \big) \rightarrow {\bf A}_{n-4}\big(\MMM_{0, K}\big) \rightarrow 0. 
\end{align*}
\qed  
\medskip

   Now we know that  $\alpha:\;{\bf A}^{*}_{K+A, comb} \rightarrow {\bf A}^{*}_{K+A, alg}$ is surjective.

\begin{proposition}\label{FirstUniv}
Let $\Gamma$ be a graph $Vert(\Gamma)=[n]$ which equals a tree with one extra edge. Assume that the unique cycle in $\Gamma$ has the odd length. Then  the set of perfect cycles $\{(ij)\}$ corresponding to the edges  of $\Gamma$ is a basis of the (free abelian) group ${\bf A}^{1}_{univ}$.
\end{proposition}
\begin{proof}
Any element of the group ${\bf A}^{1}_{univ}$ by definition has a form $\sum_{ij} a_{ij} \cdot (ij)$ with the sum ranges over all edges of the complete graph on the set $[n]$. The four-term relation can be viewed as an alternating relation for a four-edge cycle. One concludes that analogous alternating relation holds for each cycle of even length. Example: $(ij)-(jk)+(kl)-(lm)+(mp)-(pi)=0$. Such a cycle may have repeating vertices. Therefore, if a graph has an even cycle, the perfect cycles associated to its edges are dependant. 

It remains to observe that  the graph $\Gamma$ is a maximal graph without even cycles.
\end{proof}

By Theorem \ref{NewChowRing}, the Chow rings of the compactifications corresponding to complexes $K$, $K+A$, and $K+B$ are related in the following way:

$${\bf A}^{*}_{K} \cong \frac{ {\bf A}^{*}_{K+A, alg}[T]}{\big( f_{A}(T), T\cdot \ker (i_{A}^{*})\big)} \cong \frac{ {\bf A}^{*}_{K+B}[S]}{\big( f_{B}(S), S\cdot \ker (i_{B}^{*})\big)}.$$

Now we need an explicit description of the polynomials $f_{A}$ and $f_{B}$. 

Assuming that $A=\{x, x_{2}, \dots, x_{a}\}$ $B=\{y, y_{2}, \dots, y_{b}\}$, where $|A|= a$ and  $|B| = b$, take the generators
\begin{align*}
&\big\{(xy); (xy_{i}), i \in \{2, \dots, b\}; (x_{j}y), j \in \{2, \dots, a\}; (yy_{2}) \big\} \text{ for } {\bf A}^{*}_{K+ B} \text{ , and }\\
&\big\{(xy); (xy_{i}), i \in \{2, \dots, b\}; (x_{j}y), j \in \{2, \dots, a\}; (xx_{2}) \big\} \text{ for } {\bf A}^{*}_{K+ A, comb}.
\end{align*}

%\textbf{HERE YOU ASSUME THAT NO 2-ELEMENT SET IS LONG. WHAT iF THIS IS NOT THE CASE? PLEASE WRITE}

\begin{figure}[t] 
\includegraphics[width=8cm]{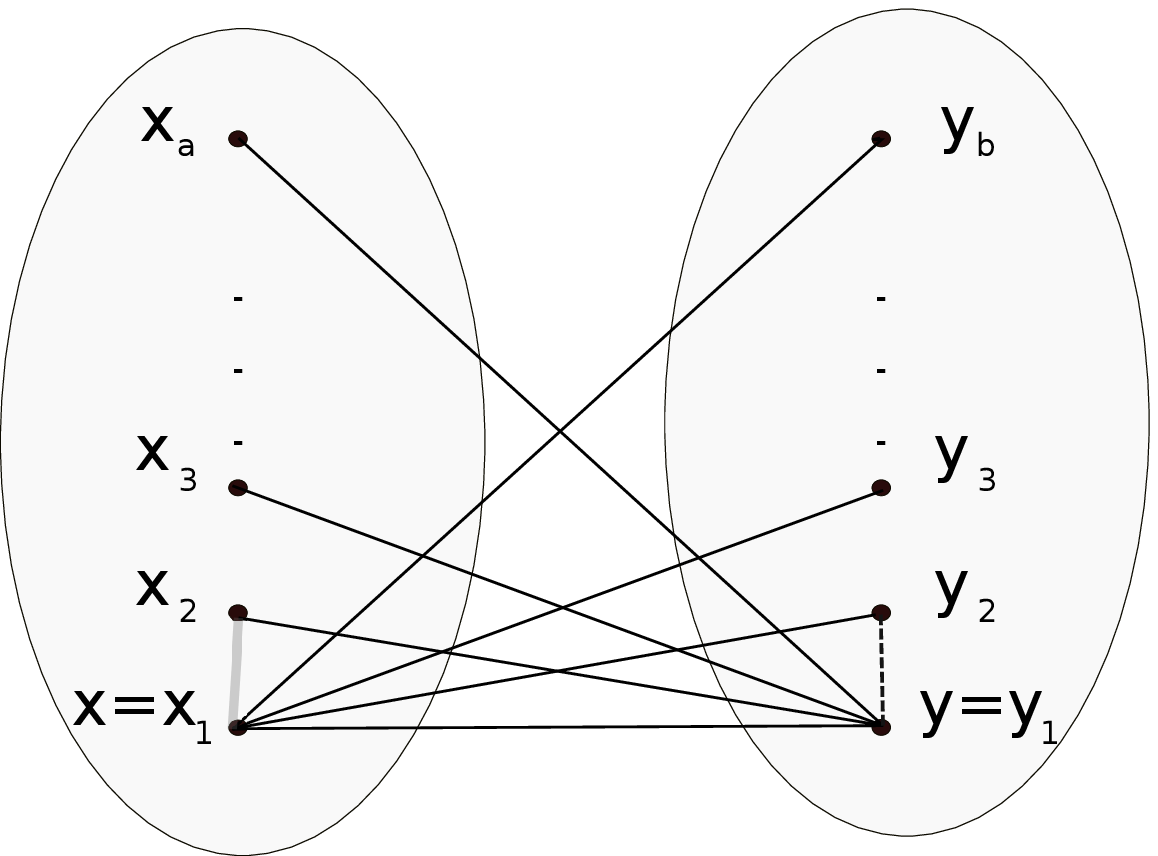}
\caption{111}
\label{Figbases}
\centering
\end{figure}

Denote by $\AA$ the subring of the Chow rings ${\bf A}^{*}_{K+ A, comb}$ and ${\bf A}^{*}_{K+ B}$ generated by the elements $\{ (xy); (xy_{i}), i \in \{2, \dots, b-1\}; (x_{j}y), j \in \{2, \dots, a-1\} \}$.

Then ${\bf A}^{*}_{K+ A, comb}$ is isomorphic to $\AA[I]/F_{B}(I)$ where $I:=(xx_{2})$ and $F_{B}(I)$ is an incarnation of the expression $(B) = (yy_{2})\cdot\dots\cdot(yy_{b})=0$ via the generators. Analogously, ${\bf A}^{*}_{K+ B, comb} \cong \AA[J]/F_{A}(J)$ with $V:= (yy_{2})$.

{The cycles $(A)$ and $(B)$ equal to the complete intersection of divisors $(xx_{2}), (xx_{3}), \dots, (xx_{a-1})$ and $(yy_{2}), (yy_{3}), \dots, (yy_{b-1})$ respectively. So  the Chern polynomials are:} $$f_{A}(T) = \big(T + (xx_{2})\big)\cdot\dots\cdot\big(T + (xx_{a-1})\big) \text{ and }f_{B}(S) = \big(S + (yy_{2})\big)\cdot\dots\cdot\big(S + (yy_{b-1})\big).$$ 

Moreover, the new variables $T$ and $S$ correspond to one and the same exceptional divisor  $\D_{A}=\D_{B}$. 

Relation between polynomials $f_{\bullet}$ and $F_{\bullet}$ are clarified in the following lemma.

\begin{lemma}\label{PullingUp}
The Chow class of the image of a divisor $(ab)_{K+A}$, $a,b \in [n]$ under the morphism $\pi_{A}^{*}$ equals
\begin{enumerate}
\item $(ab)_{K}$ for $a \in A, b\in B$, or vice versa;

\item ${\bf bl}_{(ab)(A)}\big((ab)_{K+A}\big)$ for $\{a,b\} \subset B$;

\item ${\bf bl}_{(A)}\big((ab)_{K+A}\big) + \D_{A}$ for $\{a,b\} \subset A$.
\end{enumerate}
\end{lemma}
\begin{proof}
 In case (1), the cycle $(ab)_{K+A}$ does not intersect $(A)_{K+A}$. It is by definition $\bl_{(ab)\cap(A)}\big((ab)_{K+A}\big)$. Then (1) and (2) follow directly from Theorem \ref{CorCodim},(2) by dimension counts.

The claim (3) also follows from the blow-up formula Theorem \ref{CorCodim}: 
$$\pi_{A}^{*} (ab) = \bl_{(ab)\cap (A)}\big((ab)_{K+A}\big) + j_{A, *}\big( g^{*}_{A} [(ab)\cap(A)] \cdot s(N_{\D_{A}}\MMM_{0, K}) \big)_{n-4},$$
where $N_{\D_{A}}\MMM_{0, K}$ is a normal bundle and $s(\;)$ is a total Segre class. This follows from the equalities by the functoriality of the total Chern and Segre classes and the equality $s(U)\cdot c(U) = 1$. Namely, we have
\begin{align*}
s((ab)\cap(A)) &= [(ab)\cap(A)] \cdot s(N_{(ab)\cap(A)}(ab)_{K+A}) = [(ab)\cap(A)]\cdot s\left(N_{(A)}\MMM_{0, K+A}\right), \\
c\left( \frac{g^{*}_{A}N_{(A)}\MMM_{0,K+A}}{N_{\D_{A}}\MMM_{0,K}}  \right) &\cdot g^{*}_{A} [(ab)\cap(A)]\cdot g^{*}_{A}s\left(N_{(A)}\MMM_{0, K+A}\right) = g^{*}_{A} [(ab)\cap(A)] \cdot s(N_{\D_{A}}\MMM_{0, K}).
\end{align*}

Finally, we note that $g^{*}_{A} [(ab)\cap(A)] = \D_{A}$.
\iffalse, and for the desired dimension $n-4$ we should take ``1'' from the total Segre class $s(N_{\D_{A}}\MMM_{0, K})$.  
\textbf{GAIANE DOES NOT LIKE THIS SENTENCE}
\fi
\end{proof}

From Lemma \ref{PullingUp}, we have the equality
$$f_{A}(T) = \pi_{A}^{*}\big( (xx_{2})\cdot\dots\cdot(xx_{a-1}) \big) = \pi_{A}^{*}(A)\text{ and }f_{B}(S) = \pi_{B}^{*}(B).$$

\iffalse
The following lemma states that combinatorial annihilators of $(A)$ and $(B)$ coincide.
\fi
\begin{lemma}\label{kernels}$\;$
\begin{enumerate}
\item  The ideal $\Ann(A)_{comb}$ is generated by its first graded component. 

\item  More precisely,  the generators of the ideal $\Ann(A)_{comb}$ are
\begin{enumerate}
\item the elements of type $(ab)$ with $a\in A, b\in B$, and

\item the elements of type $$(a_1a_2)-(b_1b_2),$$
where $a_1,a_2 \in A$; $b_1,b_2 \in B=A^c$.
\end{enumerate}

\item The annihilators $\Ann(A)_{comb}$ and $\Ann(B)$ are canonically isomorphic.
\end{enumerate}
\end{lemma}
\begin{proof}

First observe that the kernel $\ker(i_{A}^*)$ equal the annihilator of the cycle $(A)$. Set $\kappa$ be the ideal generated by  $\ker (i_{A}^{*})\cap A^1$. Without loss of generality, we may assume that $A=\{1,2,...,m\}$. Observe that:
\begin{enumerate}
\item  If $I\subset [n]$ has a nonempty intersection with both $A$ and $A^c$, then $(A)(I)=0$. In this case, $(I)$ can be expressed as $(ab)(I')$, where $a\in A,\ b \in A^c$.

\item \begin{enumerate}
\item[(i)] If $I\subset A$, then
$(A)\smile (I)= (A)(m+1...m+|I|).$
\item[(ii)] In this case, the element $(I)-(m+1,...,m+|I|)$ is in $\kappa$.
\end{enumerate} 

\medskip

Let us demonstrate this by giving an example with $A=\{1,2,3,4\}$, $(I)=(12)$: \newline
$(1234)\smile (12)=(A)\smile ((15)+
(16)-(56))= 0+0-(1234)(56)$.
We conclude that 
$(12)-(56)\in \kappa$.

Let us show that
$(123)-(567)\in \kappa$.

Indeed, $(123)-(567)=(12) \smile (23)-(567)\in \kappa \Leftrightarrow (56)\smile(23)-(567) \in \kappa \Leftrightarrow  (56)\smile ((23)-(67))\in \kappa$. 
Since $(23)-(67)\in \kappa$, the claim is proven.

\item If  $I\subset A^c$, then
$(A)\smile (I)= (A)(m+1,...,m+|I|).$
The element $(I)-(m+1,...,m+|I|)$ is in $\kappa$.

This follows from (1) and (2).
\end{enumerate}

Now let us prove the lemma.
  Assume $x\smile (A)=0$. Let  $x=\sum_i a_i (I_1^i)...(I^i_{k_i}).$ 

We may assume that $x$
is a homogeneous element. Modulo $\kappa$,
each summand  $(I_1)...(I_{k_i})$ can be reduced to some
$ (m+1...m+r_1)(m+r_1+1,...,m+r_2)...(m+r_{k_i}+1...m')$. Modulo $\kappa$,
$ (m+1...m+r_1)(m+r_1+1...m+r_2)...(m+r_{k_i}+1,...,m')$ can be reduced to a one-bracket element $ (m+1...m+r_1m+r_1+1...m+r_2...m+r_{k_i}+1,...,m'')$.

Indeed, for two brackets we have: $$ (m+1...m+r_1)(m+r_1+1,...,m+r_2) \equiv (m+1...m+r_1)(m+r_1,...,m+r_2-1) \equiv (m+1...m+r_2-1)\;\;(\mathrm{mod} \; \kappa).$$ For a bigger number of brackets, the statement follows by induction.

We conclude that a homogeneous $x\in \ker(i_{A}^*)$ modulo $\kappa$ reduces to some $a(m+1...m+m')$, where $a \in \mathbb{Z}$.  Then $a=0$. Indeed,
$(A)(m+1...m+m')\neq 0$ since  by Lemma \ref{LemmaTriple} $(A)(m+1...m+m')(m+m'...n)\neq 0$.
\end{proof}
\begin{remark}
Via  the four-term relation any element from  b) can be expressed as a linear combination of elements from a). So only a)--elements are sufficient to generate the annihilators. Actually, $$\AA \cong \mathbb{C} \oplus \Ann(A)_{comb}.$$ 
\end{remark}

%\medskip

We arrive at 
the following commutative diagram of graded rings:
\begin{center}
\begin{tikzcd}[column sep=small]
& {\bf A}^{*}_{K}\cong \widetilde{{\bf A}^{*}_{K}}/\Ann(B) \cong \widetilde{{\bf A}^{*}_{K}}/\Ann(A)_{comb} & \\
& \widetilde{{\bf A}^{*}_{K}}:= {\bf A}^{*}_{K+A, comb} [T]/f_{(A)}(T) \cong {\bf A}^{*}_{K+B} [S]/f_{(B)}(S) \arrow[u] & \\
& &\\
 {\bf A}^{*}_{K+B} \cong \AA[J]/F_{A}(J)\arrow[uur, "f_{B} = \pi_{B}^{*} F_{B}"] &  & {\bf A}^{*}_{K+A, comb}\cong \AA[I]/F_{B}(I) \arrow[uul, "f_{A} = \pi_{A}^{*} F_{A}"] \\
& \AA \arrow[ul, "F_{A}"] \arrow[ur, "F_{B}"] &
\end{tikzcd}
\end{center}

Therefore, the following diagram commutes:

\begin{center}
\begin{tikzcd}[column sep=small]
0\arrow[r] & \Ann(A)_{comb} \arrow[r] \arrow[d]& {\bf A}^{*}_{K+A, comb} [f_{(A)}] \arrow[r]\arrow[d, two heads, "\alpha"] & \sfrac{{\bf A}^{*}_{K+A, comb}[f_{(A)}]}{\Ann(A)_{comb}} \arrow[r] \arrow[d, "\cong"] & 0\\
0 \arrow[r]& \Ann(A)_{alg} \arrow[r] & {\bf A}^{*}_{K+A, alg} [f_{(A)}]\arrow[r] & \sfrac{{\bf A}^{*}_{K+A, alg} [f_{(A)}]}{\Ann(A)_{alg}} \arrow[r] & 0
\end{tikzcd}
\end{center}
Here $R[g]$  denotes the
 extension of a ring $R$ by a polynomial $g(t)$. All three vertical maps are induced by the map $\alpha$; the last vertical map is an isomorphism since both rings are isomorphic to ${\bf A}^{*}_{K}$.

The ideals $\Ann(A)_{comb}$ and $\Ann(A)_{alg}$ coincide,
so  the homomorphism $$\alpha:\;{\bf A}^{*}_{K+A, comb} [f_{(A)}] \rightarrow {\bf A}^{*}_{K+A, alg} [f_{(A)}]$$ is {injective}, and the theorem is proven. \qed

\iffalse

\begin{lemma}
The map $\alpha:\; {\bf A}^{*}_{K+A, comb} \rightarrow {\bf A}^{*}_{K+A, alg}$ is an isomorphism.
\end{lemma}

The surjective homomorphism $\alpha:\; {\bf A}^{*}_{K+A, comb} \rightarrow {\bf A}^{*}_{K+A, alg}$ induces an isomorphism ${\bf A}^{*}_{K+A, comb} [f_{(A)}] \rightarrow {\bf A}^{*}_{K+A, alg} [f_{(A)}]$. We get the result from this applying the Snake Lemma .
\fi

\section{Poincaré polynomials of ASD compactifications}

\begin{theorem}\label{PoiASD}
Poincaré polynomial $\PP\big(\MMM_{0, L}\big)$ for an ASD complex $L$ equals
$$\PP\big(\MMM_{0, L}\big) = \frac{1}{q(q-1)} \left( (1+q)^{n-1} - \sum\limits_{I \in L} q^{|I|} \right).$$
\end{theorem}

\begin{proof}
This theorem is proven by  Klyachko \cite[Theorem 2.2.4]{Kl} for polygon spaces, that is, for compactifications coming from a threshold ASD complex.
Assume that $K + A$ be a threshold  ASD complex. 
For the blow up of the space $\MMM_{0, K+B}$ along the subvariety $(B)$ we have an exact sequence of Chow  groups
$$0 \rightarrow {\bf A}_{p}\big(\MMM_{0, \mathcal{P}_{|A|+1}}\big) \rightarrow {\bf A}_{p}\big( \MMM_{0, \mathcal{P}_{|A|+1} } \times \MMM_{0, \mathcal{P}_{|B|+1} } \big) \oplus {\bf A}_{p}\big( \MMM_{0, K + B} \big) \rightarrow {\bf A}_{p}\big(\MMM_{0, K}\big) \rightarrow 0,$$
where, as before, $A \sqcup B = [n]$ and $p$ is a natural number. We get the equality
$$\PP(K) = \PP(K+B) + \PP(\mathcal{P}_{|A|+1})\cdot\PP(\mathcal{P}_{|B|+1})- \PP(\mathcal{P}_{|A|+1}).$$
 Also,  we have the recurrent relations for the  Poincaré polynomials:
\begin{align*}
\PP(K + B) &= \PP(K + A) + \PP(\mathcal{P}_{|A|+1}) - \PP(\mathcal{P}_{|B|+1}) =\\
&=\frac{1}{q(q-1)}\left((1+q)^{n-1}  - \sum\limits_{I\in K + A} q^{|I|} + q^{|A|} - q - q^{|B|} +q \right) = \\
&=\frac{1}{q(q-1)}\left((1+q)^{n-1}  - \sum\limits_{I\in K} q^{|I|} -q^{|A|} + q^{|A|} - q - q^{|B|} +q \right) = \\
&= \frac{1}{q(q-1)} \left( (1+q)^{n-1} - \sum\limits_{I \in K + B} q^{|I|} \right).
\end{align*}
We have used the following: if $U$ is a facet of an ASD complex, then there is an isomorphism $\MMM_{0, \mathcal{P}_{|U|+1}} \cong \mathbb{P}^{|U| - 2}$ (see Example \ref{example}(2) );  the Poincaré polynomial of the projective space $\mathbb{P}^{|U| - 2}$ equals $\frac{q^{|U|} - q}{q(q-1)}$.
\end{proof}

\section{The tautological line bundles over $\MMM_{0,K}$ and the  $\psi$--classes}\label{psi}

  The \textit{tautological line bundles} $L_i, \ i=1,...,n$  were introduced by M. Kontsevich \cite{Kon}
for  the Deligne-Mumford compactification. 
The  first Chern classes of $L_i$ are  called  the $\psi$-\textit{classes}.

We now mimic  the Kontsevich's original definition for ASD compactifications.
Let us  fix an ASD complex $K$ and the corresponding compactification $\MMM_{0,K}$.
\begin{definition} The line bundle $E_i=E_{i}(L)$ is the complex line bundle over the space $\MMM_{0,K}$  whose fiber over a point $(u_1,...,u_n)\in (\mathbb{P}^1)^n$ is the  tangent line\footnote{In the original Kontsevich's definition, the fiber over a point is the cotangent line, whereas we have the tangent line. This replacement does not create much difference.} to the projective line $\mathbb{P}^1$ at the point $u_i$.  The  first Chern class of $E_{i}$   is called the $\psi$-class and is denoted by $\psi_i$.
\end{definition}

\begin{proposition} \label{Ch2} 
\begin{enumerate}
\item For any $i \neq j \neq k\in [n]$ we have
$$\psi_i= (ij) +(ik)-(jk).$$
\item The four-term relation holds true:

$(ij)+(kl)=(ik)+(jl)$ for any distinct $i,j,k,l \in [n].$
\end{enumerate}

\end{proposition} 
\begin{proof}
(1) Take a stable configuration $(x_1,...,x_n)\in \MMM_{0,K}$. Take the circle passing through $x_i,x_j$, and $x_k$. It is oriented  by the order $ijk$. Take the  vector lying in the tangent complex line to $x_i$  which is tangent to the circle and points in the direction of $x_j$. It gives rise to a  section of $E_{i}$  which is defined correctly whenever the points $x_i,x_j$, and $x_k$ are distinct. Therefore, $\psi_i = A(ij) +B(ik)+C(jk)$ for some integer $A,B,C$. Detailed analysis specifies their values.

Now (2) follows since  the Chern class $\psi_i $ does not depend on the choice of $j$ and $k$.
\end{proof}

Let us denote by $|d_{1}, \dots, d_{n}|_K$ the intersection number $\langle \psi_1^{ d_1} ... \psi_k^{ d_k}\rangle_K=  \psi_1^{\smile d_1} \smile ... \smile \psi_k^{\smile d_k}$  related to the ASD complex $K$.

\begin{theorem}\label{ThmRecursion} Let   $\MMM_{0,K}$  be an ASD compactification.
A recursion for the intersection numbers is
\begin{align*}
|d_{1}, \dots, d_{n}|_{K} = &|d_{1}, \dots, d_{i}+d_{j}-1, \dots, \hat{d_{j}}, \dots, d_{n}|_{K_{(ij)}} + |d_{1}, \dots, d_{i}+d_{k}-1, \dots, \hat{d_{k}}, \dots, d_{n}|_{K_{(ik)}} \\
&-|d_{1}, \dots, d_{i}-1 ,\dots, d_{j}+d_{k}, \dots, \hat{d_{j}}, \hat{d_{k}},\dots, d_{n}|_{K_{(jk)}},
\end{align*}
where $i, j, k \in [n]$ are distinct.

Remind that $K_{(ij)}$ denotes the complex $K$ with $i$ and $j$ frozen together.  
 {Might happen that $K_{(ij)}$  is ill-defined, that is, $(ij)\notin K$. Then we set the corresponding summand  to be zero.}
\end{theorem}
\begin{proof}
By Proposition \ref{Ch2},
$$\langle \psi_1^{d_1}  \dots \psi_{n}^{d_n} \rangle_K = \langle \psi_1^{d_1-1} \dots \psi_{n}^{d_n} \rangle_K \smile \big((1i)+(1j)- (ij)\big).$$
It remains to observe that $\langle \psi_1^{d_1-1} \dots \psi_{n}^{d_n} \rangle_K\smile (ab)$ equals the $\langle \psi_1^{d_1-1} \dots \psi_{n}^{d_n} \rangle_{K_{(a,b)}}$.\end{proof}

\begin{theorem}\label{main_theorem} Let   $\MMM_{0,K}$  be an ASD compactification. Any top monomial in $\psi$-classes modulo renumbering has a form  $$\psi_1^{d_1}\smile ...\smile \psi_m^{d_m}$$ with $\sum_{q=1}^m d_q=n-3$  and $d_q \neq 0$ for $q=1,...,m$. Its value equals the signed number of partitions $$[n-2]=I\cup J$$ with $m+1 \in I$ and $I,J \subset K$. Each partition is counted with the sign
	
	$$(-1)^N \cdot \varepsilon ,$$
	
	where  $$N= |J|+\sum_{q \in J , q\leq m} d_q, \;\;\;\;\;\;\;\;\;\;\;\;\;\;\;\;\;\;\;\;\;\varepsilon= 
    \left\{
	\begin{array}{lll}
	1, & \hbox{if \ \  } J\cup \{n\}\in K, \hbox{and\ \  }J\cup \{n-1\}\in K;\\
    -1, & \hbox{if \ \  } I\cup \{n\}\in K, \hbox{and\ \  }I\cup \{n-1\}\in K;\\
	0, & \hbox{otherwise.}
	\end{array}
	\right.$$
	
\end{theorem}

\textit{Proof} goes by induction.

Although \textbf{the base} is trivial, let us look at it. The smallest $n$ which makes sense is $n=4$. There exist two ASD complexes with four vertices, both are threshold. So there exist two types of fine moduli compactifications, both correspond to the configuration spaces of some flexible four-gon.
The top monomials are the first powers of the $\psi$-classes.

\begin{enumerate}
  \item For $l_1=1;\ l_2=1;\ l_3=1;\ l_4=0,1$  we have
 $\psi_1=\psi_2=\psi_3=0$,  and \ \ $\psi_4=2$.

Let us prove that the theorem holds for the monomial  $\psi_1$.
There are two partitions of \newline $[n-2]~=~[2]$: 
\begin{enumerate}
\item $J=\{1\},\ I=\{2\}$.
Here $\varepsilon =0$, so this partition contributes $0$. 
\item  $J=\emptyset,\ I=\{1,2\}$. 
Here $I\notin K $, so this partition also contributes $0$. 
\end{enumerate}

   \item For $l_1=2,9;\ l_2=1;\ l_3=1;\ l_4=1$, we have
  $\psi_2=\psi_3=\psi_4=1$, and \ \ $\psi_1=-1$.
Let us check that the theorem holds for the monomial  $\psi_1$. (The other monomials are checked in a similar way.)
There  partitions of $[2]$ are the same: 
\begin{enumerate}
\item $J=\{1\},\ I=\{2\}$.
Here $\varepsilon =-1, \ N=1+1$, so this partition contributes $-1$. 
\item  $J=\emptyset,\ I=\{1,2\}$. 
Here $I\notin K $, so it contributes $0$.
\end{enumerate}

\end{enumerate}

\medskip
For the \textbf{induction step}, let us use the recursion. We shall show  that for any partition $[n-2]=I\cup J$, its contribution to the left hand side and the right hand side of the recursion are equal.

This is done through a case analysis. We present here three cases; the rest are analogous.

\begin{enumerate}
\item Assume that $i,j,k \in I$, and $(I,J)$ contributes $1$ to the left hand side count. Then

\begin{itemize}
\item[$\triangleright$] $(d_{1}, \dots, d_{i}+d_{j}-1, \dots, \hat{d_{j}}, \dots, d_{n})_{K_{(ij)}}$
contributes $1$ to the right hand side. 

Indeed, neither $N$, nor $\varepsilon$ changes  when we pass from $K$ to $K_{(ij)}$.

\item[$\triangleright$] $(d_{1}, \dots, d_{i}+d_{k}-1, \dots, \hat{d_{k}}, \dots, d_{n})_{K_{(ik)}} $ contributes $1$, and

\item[$\triangleright$] $-(d_{1}, \dots, d_{i}-1 ,\dots, d_{j}+d_{k}, \dots, \hat{d_{j}}, \hat{d_{k}},\dots, d_{n})_{K_{(jk)}}$ contributes $-1$.
\end{itemize}

\item Assume that $i\in I,\ j,k \in J$, and $(I,J)$ contributes $1$ to the left hand side count. Then

\begin{itemize}
\item[$\triangleright$] $(d_{1}, \dots, d_{i}+d_{j}-1, \dots, \hat{d_{j}}, \dots, d_{n})_{K_{(ij)}}$
contributes $0$ to the right hand side. 

\item[$\triangleright$] $(d_{1}, \dots, d_{i}+d_{k}-1, \dots, \hat{d_{k}}, \dots, d_{n})_{K_{(ik)}} $ contributes $0$, and

\item[$\triangleright$] $-(d_{1}, \dots, d_{i}-1 ,\dots, d_{j}+d_{k}, \dots, \hat{d_{j}}, \hat{d_{k}},\dots, d_{n})_{K_{(jk)}}$ contributes $1$.

Indeed, $N$ turns to $N-1$, whereas $\varepsilon$ stays the same.
\end{itemize}

\item  Assume that $i\in J,\ j,k \in I$, and $(I,J)$ contributes $1$ to the left hand side count. Then

\begin{itemize}
\item[$\triangleright$] $(d_{1}, \dots, d_{i}+d_{j}-1, \dots, \hat{d_{j}}, \dots, d_{n})_{K_{(ij)}}$
contributes $0$. 

\item[$\triangleright$] $(d_{1}, \dots, d_{i}+d_{k}-1, \dots, \hat{d_{k}}, \dots, d_{n})_{K_{(ik)}} $ contributes $0$, and

\item[$\triangleright$] $-(d_{1}, \dots, d_{i}-1 ,\dots, d_{j}+d_{k}, \dots, \hat{d_{j}}, \hat{d_{k}},\dots, d_{n})_{K_{(jk)}}$ contributes $1$, since $N$ turns to $N-1$, and $\varepsilon$ stays the same.\qed
\end{itemize}\end{enumerate} 
  
This theorem was proven for polygon spaces (that is, for threshold ASD complexes) in \cite{Agapito}.

\section{Appendix. Chow rings and blow ups}

Assume we have a diagram of a blow up $\widetilde{Y}:={\bf bl}_{X}(Y)$. Here  $X$ and $Y$ are smooth varieties, $\iota: X \hookrightarrow Y$ is a regular embedding, and  $\widetilde{X}$ is the exceptional divisor. In this case, $\iota^{*}: A^*(Y) \rightarrow A^*(X)$ is surjective. 
\begin{center}
\begin{tikzcd}
\widetilde{X} \arrow{d}{\tau} \arrow[r, hookrightarrow, "\theta"] & \widetilde{Y} \arrow{d}{\pi} \\
X  \arrow[r, hookrightarrow, "\iota"] & Y
\end{tikzcd}
\end{center}
Denote by $E$  the relative normal bundle $$E:= \tau^* N_X Y/ N_{\widetilde{X}} \widetilde{Y}.$$

\begin{theorem}\cite[Appendix. Theorem 1]{Keel92}\label{NewChowRing}
The Chow ring $A^{*}(\widetilde{Y})$ is isomorphic to
$$\frac{A^{*}(Y)[T]}{(P(T), T\cdot \ker i^*)},$$
where $P(T) \in A^*(Y)[T]$ is the pullback from $A^*(X)[T]$ of Chern polynomial of the normal bundle $N_{X}Y$. This isomorphism is induced by $\pi^* : A^*(Y)[T] \rightarrow A^* (\widetilde{Y})$ which sends $-T$ to the class of the exceptional divisor  $\widetilde{X}$.
\end{theorem}

\begin{theorem}\cite[Proposition 6.7]{FulInter}\label{SESBlow} Let $k \in \mathbb{N}$.
\begin{itemize}
\item[a)](Key formula) For all $x \in A_{k}(X)$
$$\pi^* \iota_{*} (x) = \theta_* (c_{d-1}(E)\cap \tau^* x)$$
\item[e)] There are split exact sequences
$$0 \rightarrow A_{k} X \xrightarrow{\upsilon} A_{k} \widetilde{X} \oplus A_{k} Y \xrightarrow{\eta} A_{k} \widetilde{Y} \rightarrow 0$$
with $\upsilon(x) = \big( c_{d-1}(E)\cap \tau^{*}x, -\iota_{*}x \big)$, and $\eta(\tilde{x}, y) = \theta_{*}\tilde{x}+ \pi^* y$. A left inverse for $\upsilon$ is given by $(\tilde{x}, y) \mapsto \tau_{*}(\tilde{x})$.
\end{itemize}
\end{theorem}

\begin{theorem}\cite[Theorem 6.7, Corollary 6.7.2]{FulInter}\label{CorCodim} $\;$
\begin{enumerate}
\item (Blow-up Formula) Let $V$ be a $k$-dimensional subvariety of $Y$, and let $\tilde{V} \subset \tilde{Y}$ be the proper transform of $V$, i.e., the blow-up of $V$ along $V\cap X$. Then $$\pi^{*}[V] = [\tilde{V}] + j_{*}\{c(E)\cap \tau^{*} s(V\cap X, V)\}_{k} \text{ in }A_{k}\tilde{Y}.$$
\item If $\dim V\cap X \leq k-d$, then $\pi^* [V] = [\tilde{V}]$.
\end{enumerate}
\end{theorem}

An algebraic variety $Z$ is a \textit{HI--scheme} if the canonical map $\mathrm{cl}: A^{*}(Z) \rightarrow H^{*}(Z, \mathbb{Z})$ is an isomorphism.

\begin{theorem}\cite[Appendix. Theorem 2]{Keel92}\label{ChowCoh}
\begin{itemize}
\item If $X, \widetilde{X}$, and $Y$ are HI, then so is $\widetilde{Y}$.

\item If $X, \widetilde{X}$, and $\widetilde{Y}$ are HI, then so is $Y$.
\end{itemize}
\end{theorem}

\vspace{1cm}

\textbf{Acknowledgement.} This research is supported by the Russian Science Foundation under grant 16-11-10039. %\textcolor{blue}{I.N. thanks «Native towns», a social investment program of PJSC «Gazprom Neft».}


\begin{thebibliography}{}

\bibitem{Smyth13} Smyth, D.I. Invent. math. (2013) 192: 459. \url{https://doi.org/10.1007/s00222-012-0416-1}.

\bibitem{HKn} J.-C. Hausmann, A. Knutson, \emph{The cohomology ring of polygon spaces}, Annales de l'institut Fourier  48, 1 (1998), 281--321.


\bibitem{NP} I. Nekrasov, and G. Panina. Geometric presentation for the cohomology ring of polygon spaces, arXiv:1801.00785.


\bibitem{Kon} M. Kontsevich \emph{Intersection theory on the moduli space of curves and the matrix Airy function}, Comm. Math. Phys., 147, 1(1992), 1--23.

\bibitem{PanGal}  P. Galashin, G. Panina \emph{Manifolds associated to simple games},
J. Knot Theory Ramifications, 25(12):1642003, 14, 2016.

%\bibitem{fa} M. Farber, \textit{Invitation to Topological Robotics}, Zuerich Lectures in Advanced Mathematics, European Mathematical Society (EMS), Zuerich, 2008. 
\bibitem{DM69} Deligne, P. \& Mumford, D. Publications Mathématiques de l’Institut des Hautes Scientifiques (1969) 36: 75--109. \url{https://doi.org/10.1007/BF02684599}.

\bibitem{Keel92} Keel, Sean. “Intersection Theory of Moduli Space of Stable N-Pointed Curves of Genus Zero.” Vol. 330, no. 2, 1992, pp. 545–574., \url{https://doi.org/10.2307/2153922 }.

\bibitem{Has03} Hassett, Brendan. Moduli spaces of weighted pointed stable curves, Advances in Mathematics, Volume 173, Issue 2, 2003, Pages 316-352, ISSN 0001-8708, \url{https://doi.org/10.1016/S0001-8708(02)00058-0}.


\bibitem{Cey07}  Ceyhan, Özgür. Chow groups of the moduli spaces of weighted pointed stable curves of genus zero, Advances in Mathematics, Volume 221, Issue 6, 2009, Pages 1964-1978, ISSN 0001-8708, \url{https://doi.org/10.1016/j.aim.2009.03.011}.


\bibitem{Kl} A. Klyachko, \emph{Spatial polygons and stable configurations of points in the projective line}, In: Tikhomirov A., Tyurin A. (eds) Algebraic Geometry and its Applications. Aspects of Mathematics, vol 25 (1994), 67-84.

\bibitem{vNMor}
Von Neumann, J., Morgenstern, O.: Theory of games and economic behavior. Bull. Amer. Math. Soc 51, 498-504 (1945)

\bibitem{Han15} H.-B. Moon, C. Summers, J. von Albade, and R. Xie. Birational contractions of $\bar{M}_{0,n}$ and combinatorics of extremal assignments. J. Algebraic Comb., Vol. 47, (2018), no. 1, 51–90.

\bibitem{LosMan00} A. Losev and Y. Manin. New moduli spaces of pointed curves and pencils of flat connections, Michigan Math. J., Volume 48, Issue 1 (2000), 443-472. \url{https://doi.org/10.1307/mmj/1030132728}.
\bibitem{Agapito} J. Agapito, L. Godinho,  \emph{ Intersection numbers of polygon spaces}, Trans. Amer. Math. Soc. 361 (2009), 4969--4997.

\bibitem{FulInter} Fulton, William. Intersection Theory, Springer-Verlag New York, 1998.

%\bibitem{3264} Eisenbud D. and Harris J. 3264 \& All That Intersection Theory in Algebraic Geometry, preprint.










\end{thebibliography}
\end{document}